\documentclass[11pt,reqno]{amsart}
\usepackage[T1]{fontenc}
\usepackage{graphicx}
\usepackage{cite}

\usepackage{color}
\definecolor{MyLinkColor}{rgb}{0,0,0.4}

\newcommand{\R}{{\mathbb R}}

\newcommand{\kB}{\mathcal{B}}

\newcommand{\cO}{\mathcal{O}}

\newcommand{\kL}{\mathcal{L}}

\newcommand{\re}{\mathop{\rm Re}\nolimits}

\newcommand{\ov}{\overline}
\newcommand{\p}{\partial}

\newcommand{\0}{\Omega}

\newtheorem{thm}{Theorem}[section]

\newtheorem{lemma}[thm]{Lemma}

\theoremstyle{remark} 

\newtheorem{defin}[thm]{Definition}

\numberwithin{equation}{section}

\begin{document}

\title[Weak-strong uniqueness]{Weak-strong uniqueness for a class of degenerate parabolic cross-diffusion systems}

\thanks{}
\author{Philippe Lauren\c{c}ot}
\address{Institut de Math\'ematiques de Toulouse, UMR~5219, Universit\'e de Toulouse, CNRS \\ F--31062 Toulouse Cedex 9, France}
\email{laurenco@math.univ-toulouse.fr}
\author{Bogdan-Vasile Matioc}
\address{Fakult\"at f\"ur Mathematik, Universit\"at Regensburg \\ D--93040 Regensburg, Deutschland}
\email{bogdan.matioc@ur.de}

\begin{abstract}
Bounded weak solutions to a particular class of degenerate parabolic cross-diffusion systems are shown to coincide with the unique strong solution determined by the same initial condition on the maximal existence interval of the latter. The proof relies on an estimate established for a relative entropy associated to the system.
\end{abstract}


\maketitle

\pagestyle{myheadings}
\markboth{\sc{Ph.~Lauren\c cot \& B.-V.~Matioc}}{\sc{Weak-strong uniqueness}}

 \section{Introduction}\label{Sec:1}

Let $\Omega$ be a bounded domain of $\mathbb{R}^N$, $N\ge 1$, with smooth boundary $\partial\Omega$ and outer unit normal $\mathbf{n}$, and assume that the constants~$a,\, b,\, c,$ and $d$ satisfy 
 \begin{equation}\label{condabcd}
 	(a,\, b,\, c,\, d) \in(0,\infty)^4 \qquad\text{and}\qquad ad>bc\,.
 \end{equation}
 We consider the evolution equations 
\begin{subequations}\label{PB}
\begin{equation}\label{PB1ab}	
	\left.\begin{array}{rcrr}
		\partial_t f \!\!\!& =  &\!\!\!\mathrm{div}\big(f \nabla[ af + b g ] \big)   \\[1ex]
		\partial_t g  \!\!\!& =  &\!\!\!\mathrm{div}\big(g \nabla[ cf + d g] \big)
	\end{array}
	\right\}\;\text{ in }\; (0,\infty)\times \Omega\,,
	\end{equation}
 supplemented with homogeneous Neumann boundary conditions
\begin{equation}
	\nabla f\cdot \mathbf{n} = \nabla g\cdot \mathbf{n} = 0 \;\text{ on }\; (0,\infty)\times \partial\Omega\,, \label{PB1c}
\end{equation}
and non-negative initial conditions
\begin{equation}
	(f,g)(0) = (f^{in},g^{in}) \;\text{ in }\; \Omega\,. \label{PB1d}
\end{equation}
\end{subequations}
The porous medium equation \cite{Va07} as well as the thin film Muskat problem \cite{EMM12} arise as special cases of~\eqref{PB1ab}. 

We point out that \eqref{PB1ab} is a quasilinear degenerate parabolic system with a  full diffusion matrix, so that the study of its well-posedness is already a challenging issue. On the one hand, owing to its parabolic structure, the system~\eqref{PB} fits into the theory developed in \cite{Am93}, from which the local existence and uniqueness of a strong solution with positive components can be inferred, see Theorem~\ref{T:WP} below. However, comparison principles cannot be applied in the context of~\eqref{PB} and the degeneracy featured in~\eqref{PB1ab} might lead to the breakdown of the positivity of the components in finite time and thus to that of their regularity. As a consequence, strong solutions cannot be extended beyond a finite time in general. On the other hand, non-negative global weak solutions to~\eqref{PB}, which are also bounded, are constructed in \cite{LM22,LM22yy}, but the uniqueness of  such solutions is an open problem, even in dimension $N=1$. This is in sharp contrast with the porous  medium equation for which several uniqueness results for weak solutions are available in the literature, see \cite{AL1983, BCP1984, BC1979, 0t1996, Pi1982, Va07, LM2021a} and the references therein. It is actually the strong coupling in~\eqref{PB1ab} which makes it difficult to generalize the methods from the above references to this two-phase version of the porous medium equation.

The goal of this paper is to prove a weaker result, namely that all bounded weak solutions to~\eqref{PB} coincide on some time interval on which a strong solution exists. 
For that purpose, we shall rely on the availability of a suitable relative entropy functional, an idea which has proved instrumental in several recent works 
on weak-strong uniqueness/stability results for (systems of) partial differential equations. 
In particular, this method has been applied in various settings such as: the compressible  Navier-Stokes  system~\cite{FJN2012} and  the Fourier-Navier-Stokes  system~\cite{FN2012}, the (isentropic) Euler equations \cite{GSW15, BLS11}, hyperbolic-parabolic systems~\cite{CT2018}, the Navier-Stokes-Korteweg and the Euler-Korteweg systems~\cite{BGL2019, GLT17}, the Navier-Stokes equation  with surface tension~\cite{FH2020}, (reaction-)cross-diffusion systems~\cite{JC2019, JPZ2022}, entropy-dissipating reaction-diffusion equations~\cite{F2017},  energy-reaction-diffusion systems~\cite{H2022}, and Maxwell-Stefan systems~\cite{HJT22x}. 

Before stating precisely our main result, let us first make precise the meaning of weak and strong solutions to~\eqref{PB}. 

\begin{defin}[Bounded weak solution]\label{D:1}
	Assume \eqref{condabcd} and let $ u^{in} := (f^{in},g^{in}) $ be an element of~${L_{\infty,+}(\Omega,\mathbb{R}^2)}$. Given $T\in (0,\infty]$, a bounded weak solution $u$ to~\eqref{PB} on $[0,T)$ is a
	 pair of functions~${u=(f,g)}$ such that:
	\begin{itemize}
		\item[(i)] for each $t\in (0,T)$, 
		\begin{equation}
			(f,g)\in L_{\infty,+}((0,t)\times \Omega,\mathbb{R}^2) \cap L_2((0,t),H^1(\Omega,\mathbb{R}^2))\cap W_2^1((0,t),H^1 (\Omega,\mathbb{R}^2)')\,; \label{p1}
		\end{equation}
		\item[(ii)] for all $\varphi\in H^1(\Omega)$ and $t\in (0,T)$,
		\begin{subequations}\label{p2}
			\begin{equation}
				\int_\Omega (f(t,x)-f^{in}(x)) \varphi(x)\ \mathrm{d}x  +\int_0^t \int_\Omega f(s,x) \nabla[ a f + b g ](s,x) \cdot \nabla\varphi(x)\, \mathrm{d}x\mathrm{d}s =0 \label{p2a}
			\end{equation} 
			and
			\begin{equation}
				\int_\Omega (g(t,x)-g^{in}(x)) \varphi(x)\ \mathrm{d}x +  \int_0^t \int_\Omega g(s,x) \nabla[cf + dg ](s,x) \cdot \nabla\varphi(x)\, \mathrm{d}x\mathrm{d}s =0\,. \label{p2b}
			\end{equation}
		\end{subequations} 
	\end{itemize}
\end{defin} 
Observe that the boundedness and weak differentiability required on $f$ and $g$ in~\eqref{p1} guarantee that the integrals in~\eqref{p2} are finite.

We next turn to strong solutions to~\eqref{PB} and first introduce some notation: for $p>N$ and $s\in (1+N/p,2]$, we set 
\begin{equation*}
	H^s_{p,\kB}(\0):=\{z\in H^s_{p}(\0)\,:\, \text{$\nabla z\cdot\mathbf{n}=0$ on $\p\0$} \}\,,
\end{equation*}
where $H^s_{p}(\0)$ denotes the Bessel potential space, see \cite[Section~5]{Am93} for instance, and
\begin{equation}
\cO_p^s:=\{u=(f,g)\in H^s_{p,\kB}(\0,\mathbb{R}^2)\,:\, \text{$f>0$ and $g>0$ in $\Omega$}\}\,. \label{Osp}
\end{equation}
We observe that the continuous embedding of $H^s_{p}(\0)$ in ${\rm C}^1(\ov\0)$ for $p>N$ and $s\in(1+N/p,2]$ guarantees that $\cO_p^s$ is an open subset of $H^s_{p,\kB}(\0,\mathbb{R}^2)$.

\begin{defin}[Strong solution]\label{D:2}
Assume~\eqref{condabcd} and let $p>N$, $s\in (1+N/p,2)$, ${T\in (0,\infty]}$, and $u^{in}=(f^{in},g^{in})\in \cO_p^s$. A strong solution $u$ to~\eqref{PB} on $[0,T)$ is a pair~${u=(f,g)}$ such that
\begin{equation*}	
	 u\in {\rm C}([0,T), \cO_p^s)\cap {\rm C}^1((0,T), L_p(\0,\mathbb{R}^2))\cap {\rm C}((0,T),  H^2_{p,\kB}(\0,\mathbb{R}^2))\,,
\end{equation*}
which satisfies~\eqref{PB} in a strong sense (and in particular a.e. in $(0,T)\times\Omega$).
\end{defin}
One may easily check that a strong solution to~\eqref{PB} on $[0,T)$ in the sense of Definition~\ref{D:2} is also a bounded weak solution on $[0,T)$ in the sense of Definition~\ref{D:1}.

\bigskip

The aim of this paper is to establish a weak-strong uniqueness result for~\eqref{PB} as stated in Theorem~\ref{MT1} below. 
As in \cite{F2017}, the main tool to be used in the proof is the relative entropy functional
\begin{equation}\label{reen}
	H(u_1|u_2):=\int_\Omega \left\{ \bigg[f_1\ln\Big(\frac{f_1}{f_2}\Big)-(f_1-f_2)\bigg]+ \cfrac{b }{c} \bigg[ g_1 \ln\Big(\frac{g_1}{g_2}\Big) - (g_1-g_2) \bigg] \right\}\,\mathrm{d}x\,,
\end{equation}
which is well-defined for $u_i=(f_i,g_i)\in L_{2,+}(\Omega,\mathbb{R}^2)$, $i=1,\, 2$, provided that $f_2$ and $g_2$ are bounded from below by positive constants.
It is important to remark that $H(u_1|u_2)$ controls 
the square of the $L_2$-norm of $u_1-u_2$,  see~\eqref{eq:L2} below, if  $u_1$ and $u_2$ are additionally bounded functions. 

The main step in the proof of Theorem~\ref{MT1} is to derive the integral inequality~\eqref{DES} which measures the ``distance'' between a bounded weak solution in the sense of Definition~\ref{D:1} and a strong solution in the sense of Definition~\ref{D:2}. Gronwall's inequality then provides the weak-strong uniqueness property for the evolution problem~\eqref{PB}.

\begin{thm}\label{MT1}
	Consider $u_1^{in}\in L_{\infty,+}(\Omega,\mathbb{R}^2)$ and $u_2^{in}\in \mathcal{O}_p^s$ for some $s\in (1+N/p,2)$ and~${p>N}$.
	 Let $u_2=(f_2,g_2) $  be the strong solution to~\eqref{PB} with initial condition $u_2^{in}$ defined on its maximal existence interval $[0,T^+)$, $T^+\in (0,\infty]$, see Theorem~\ref{T:WP} below.
	  If~${u_1=(f_1,g_1)}$  is a bounded weak solution to~\eqref{PB} on $[0,T^+)$ with initial condition $u_1^{in}$ and~${T\in(0,T^+),}$  there exists a positive constant~${C=(a,b,c,d,u_1,u_2,T)}$ such that 
	\begin{equation}\label{DES}
		H(u_1(t)|u_2(t))\leq  H(u_1^{in}|u_2^{in})+C \int_0^t H(u_1(s)|u_2(s))\, \mathrm{d}s\qquad\text{for all $t\in[0,T]$\,.}
	\end{equation}
	In particular, if $u_1^{in}=u_2^{in}$, then $u_1(t)=u_2(t)$ for all $t\in[0,T^+)$. 
\end{thm}

We emphasize that Theorem~\ref{MT1} applies to any pair of initial conditions $u_1^{in}\in L_{\infty,+}(\Omega,\mathbb{R}^2)$ and $u_2^{in}\in \mathcal{O}_p^s$ for some $s\in (1+N/p,2)$ and $p>N$. Indeed, the existence of a bounded weak solution to~\eqref{PB} on $[0,\infty)$ with initial condition $u_1^{in}$ follows from \cite{LM22,LM22yy}, while that of a strong solution to~\eqref{PB} on some maximal time interval with initial condition $u_2^{in}$ is provided in Theorem~\ref{T:WP} below.

\section{Proof of the main result}\label{Sec:2}

We start this section by considering the evolution problem~\eqref{PB}  in the setting of strong solutions as specified in Definition~\ref{D:2}. Using the quasilinear parabolic theory developed in \cite{Am93}, we then prove in Theorem~\ref{T:WP} that \eqref{PB} is well-posed in this strong setting. The remaining part is then devoted to the proof of Theorem~\ref{MT1}. 

\subsection{Strong solutions to the evolution problem~\eqref{PB}}\label{Sec:2a}

In order to construct strong solutions to~\eqref{PB} we reformulate \eqref{PB1ab} in a suitable framework. For that purpose, we fix~$p>N$ and $s\in (1+N/p,2)$  and introduce the mobility matrix
\begin{equation}
	M(X)  =  (m_{jk}(X))_{1\le j,k\le 2} := \begin{pmatrix}
		a X_1& b X_1 \\ 
		c X_2 & d X_2
	\end{pmatrix}\,, \qquad X=(X_1,X_2)\in \mathbb{R}^2\,. \label{x1}
\end{equation}
The problem ~\eqref{PB} can then be recast as
\begin{equation}\label{NNEP}
\frac{\mathrm{d}u}{\mathrm{d}t}(t)=\Phi(u(t))[u(t)],\qquad u(0)=u^{in}\,,
\end{equation}
where the quasilinear operator $\Phi:\cO_p^s\to\kL(H^2_{p,\kB}(\0,\mathbb{R}^2), L_p(\0,\mathbb{R}^2))$ is defined by the relation 
\[
\Phi(u)[v]:=\mathrm{div}(M(u)\nabla v)=\sum_{i=1}^N\p_i(M(u)\p_i v),\qquad u\in\cO_p^s,\, v\in H^2_{p,\kB}(\0,\mathbb{R}^2)\,.
\]
Observing that, for $u\in\mathcal{O}_p^s$, the matrix-valued function $M(u)$ belongs to ${\rm C}^1(\ov\0,\R^{2\times 2})$ and that  $M(u(x))$, $x\in \ov\0$,  has its spectrum contained in the right-half plane~${\{\re z>0\}}$, we infer from \cite[Theorem~4.1  and Example~4.3~(e)]{Am93} that~$\Phi(u)$ is, for each $u\in\cO_p^s,$ the generator of an analytic semigroup in $\kL(L_p(\0,\mathbb{R}^2))$.
Since
\[
[L_p(\0), H^{2}_{p,\kB}(\0)]_{s/2}=H^s_{p,\kB}(\0)\,,
\]
 where $[\cdot,\cdot]$ is the complex interpolation functor, see \cite[Theorem~4.3.3]{Tr78}, 
 we may now apply to~\eqref{NNEP} the quasilinear parabolic theory presented in \cite[Section~12]{Am93} (see also \cite[Remark~1.2~(ii)]{MW20})
 to obtain the following result.
 \begin{thm}\label{T:WP}
 Let $p>N$,    $s\in(1+N/p,2)$, and assume that \eqref{condabcd} is satisfied.
  Then, given~${u^{in}\in\cO_p^s}$, the problem~\eqref{PB} has a unique  maximal strong solution
 \[
 u\in {\rm C}([0,T^+), \cO_p^s)\cap {\rm C}^1((0,T^+), L_p(\0,\mathbb{R}^2))\cap {\rm C}((0,T^+),  H^2_{p,\kB}(\0,\mathbb{R}^2))\,,
 \]
 where $T^+\in(0,\infty]$ denotes the maximal existence time. 
 \end{thm}
 
\subsection{Proof of Theorem~\ref{MT1}}\label{Sec:2b}

Let $T\in (0,T^+)$. Since $\{u_2(t)\ :\ t\in [0,T]\}$ is a compact subset of $\mathcal{O}_p^s$, there is $\sigma\in (0,1)$ (possibly depending on $T$) such that, for $t\in [0,T]$,
 \begin{equation}\label{bounds}
\sigma\leq \min_{x\in\overline\Omega} \min\big\{f_2(t,x),\,g_2(t,x)\big\}\quad \text{and}\quad \max\big\{\|\nabla f_2(t)\|_\infty,\,\|\nabla g_2(t)\|_\infty\big\}\leq \sigma^{-1}\,.
 \end{equation}
 Moreover, since $u_1$ is a bounded weak solution, we may assume that also  
 \begin{equation}\label{bounds2}
|u_1(t,x)|+|u_2(t,x)|\leq \sigma^{-1}\qquad\text{a.e. in $(0,T)\times\Omega$}\,.
 \end{equation}

Given $\eta\in(0,1)$, let 
\begin{align*}
H_\eta(u_1(t)|u_2(t))&:=\int_\Omega \bigg[f_1(t)\ln\Big(\frac{f_1(t)+\eta}{f_2(t)}\Big)-(f_1(t)-f_2(t))\bigg]\,\mathrm{d}x\\[1ex]
&\quad+ \cfrac{b }{c}\int_\Omega\bigg[g_1(t)\ln\Big(\frac{g_1(t)+\eta}{g_2(t)}\Big)-(g_1(t)-g_2(t))\bigg]\,\mathrm{d}x, \qquad t\in[0,T]\,.
\end{align*}
As a consequence of ~\eqref{bounds2} and of Definition~\ref{D:1}, we have $u_i(t)\in L_\infty(\Omega,\R^2)$ for $i=1,\, 2$ and all~${t\in[0,T]}$, 
and the dominated convergence theorem, together with  the lower bound in~\eqref{bounds}, yields
\begin{equation}\label{EST0}
\lim_{\eta\to0} H_\eta(u_1(t)|u_2(t)) = H(u_1(t)|u_2(t))\,, \qquad t\in [0,T]\,.
\end{equation}
Furthermore, by virtue of  Definition~\ref{D:1}, Definition~\ref{D:2},   \eqref{bounds}, \eqref{bounds2}, and the continuous embedding of $\mathcal{O}_p^s$ in ${\rm C}^1(\overline{\Omega},\mathbb{R}^2)$ we have 
\begin{equation*}
f_1,\, g_1 \in L_2((0,T), H^1(\0))\cap W^1_2((0,T), H^1(\0)') 
\end{equation*}
and
\begin{equation*}
	\ln{f_2},\, \ln{g_2} \in L_2((0,T), H^1(\0))\cap W^1_2((0,T), H^1(\0)') \,.
\end{equation*}
These  properties, together with \eqref{bounds}, \eqref{bounds2}, and suitable versions of the Lions-Magenes lemma, see, e.g., \cite[Theorem~II.5.12]{BF13} and Lemma~\ref{L:A1}, imply that 
\[
[t\mapsto H_\eta(u_1(t)|u_2(t))]:[0,T]\to\R
\]
 is continuous and
\begin{equation}\label{eq:QE}
\begin{aligned}
&H_\eta(u_1(t)|u_2(t))-H_\eta(u_1^{in}|u_2^{in})\\[1ex]
&=\int_0^t \Big\langle \p_t f_1,\ln\Big(\frac{f_1+\eta}{f_2}\Big) + \frac{f_1}{f_1+\eta}\Big\rangle_{(H^1)', H^1}\, \mathrm{d}s 
- \int_0^t\Big\langle \p_t f_2,\frac{f_1}{f_2}\Big\rangle_{(H^1)', H^1}\, \mathrm{d}s\\[1ex]
&\quad+\cfrac{b}{c} \int_0^t\Big\langle \partial_t g_1,\ln\Big(\frac{g_1+\eta}{g_2}\Big) + \frac{g_1}{g_1+\eta}\Big\rangle_{(H^1)', H^1} \, \mathrm{d}s 
- \cfrac{b}{c} \int_0^t \Big\langle \p_t f_2,\frac{f_1}{f_2}\Big\rangle_{(H^1)', H^1}\, \mathrm{d}s
\end{aligned}
\end{equation}
for all $t\in[0,T]$, where $\langle\cdot,\cdot\rangle_{(H^1)', H^1} $ is the duality bracket between $H^1(\0)$ and $H^1(\0)'$. Reformulating \eqref{eq:QE} with the help of~\eqref{p2}, we find
\begin{equation*}
\begin{aligned}
&H_\eta(u_1(t)|u_2(t))-H_\eta(u_1^{in}|u_2^{in}) \\[1ex]
&=-\int_0^t \int_\0  f_1\nabla(af_1+bg_1)\cdot\Big(\frac{\nabla f_1}{f_1+\eta}-\frac{\nabla f_2}{f_2 }\Big)\,\mathrm{d}x \mathrm{d}s\\[1ex]
&\quad-\int_0^t \int_\0  \left[ f_1\nabla(af_1+bg_1)\cdot \nabla\Big(\frac{ f_1}{f_1+\eta}\Big)-f_2\nabla(af_2+bg_2)\cdot \nabla\Big(\frac{f_1}{f_2 }\Big) \right]\,\mathrm{d}x \mathrm{d}s\\[1ex]
&\quad-\cfrac{b}{c}\int_0^t \int_\0  g_1\nabla(cf_1+dg_1)\cdot\Big(\frac{\nabla g_1}{g_1+\eta}-\frac{\nabla g_2}{g_2 }\Big)\,\mathrm{d}x \mathrm{d}s\\[1ex]
&\quad-\cfrac{b}{c}\int_0^t \int_\0 \left[ g_1\nabla(cf_1+d g_1)\cdot \nabla\Big(\frac{ g_1}{g_1+\eta}\Big) - g_2\nabla(cf_2+dg_2)\cdot \nabla\Big(\frac{g_1}{g_2 }\Big) \right]\,\mathrm{d}x \mathrm{d}s \,.
\end{aligned}
\end{equation*}
Hence,  
\begin{equation}\label{eq:QE*}
H_\eta(u_1(t)|u_2(t))-H_\eta(u_1^{in}|u_2^{in})=T^1_\eta(t)+T^2(t)\,,
\end{equation}
where
\begin{align*}
T^1_\eta(t)&:= \eta^2 \int_0^t \int_\Omega \nabla(af_1+bg_1)\cdot \frac{\nabla f_1}{(f_1+\eta)^2}\, \mathrm{d}x \mathrm{d}s\\[1ex]
&\quad +\eta^2 \cfrac{b}{c}\int_0^t \int_\Omega \nabla(cf_1+dg_1)\cdot \frac{\nabla g_1}{(g_1+\eta)^2}\, \mathrm{d}x \mathrm{d}s
\end{align*}
and 
\begin{align*}
T^2(t):=& -\int_0^t \int_\0 \left[ \nabla(af_1+bg_1)\cdot\Big( \nabla f_1 - \frac{f_1}{f_2} \nabla f_2 \Big) - f_2\nabla(af_2+bg_2)\cdot\nabla\Big(\frac{f_1}{f_2 }\Big) \right]\,\mathrm{d}x \mathrm{d}s\\[1ex]
&-\cfrac{b }{c}\int_0^t \int_\0 \left[ \nabla(cf_1+dg_1)\cdot\Big(\nabla g_1-\frac{g_1}{g_2 } \nabla g_2 \Big) -   g_2 \nabla(cf_2+dg_2)\cdot \nabla\Big(\frac{g_1}{g_2 }\Big) \right]\,\mathrm{d}x \mathrm{d}s\,.
\end{align*}
In view of Definition~\ref{D:1}~(i), both functions $\nabla(af_1+bg_1)\cdot\nabla f_1$ and $\nabla(cf_1+dg_1)\cdot\nabla g_1$ belong to~${L_1((0,t)\times\Omega)}$ and
\[
\left.
\begin{array}{r}
\displaystyle\lim_{\eta\to 0} \eta^2 \frac{\nabla f_1}{(f_1+\eta)^2} =0\\[1ex]
\displaystyle\lim_{\eta\to 0} \eta^2 \nabla\frac{ g_1}{(g_1+\eta)^2} =0 
\end{array}
\right\}\qquad\text{a.e. in $(0,t)\times\Omega$\,},
\]
as $\nabla f_1=0$ a.e. on $\{(s,x)\,:\, f_1(s,x)=0\}$ and $\nabla g_1=0$ a.e. on $\{(s,x)\,:\, g_1(s,x)=0\}.$ 
The dominated convergence theorem now implies that 
\begin{equation}\label{EST1}
\lim_{\eta\to0} T^1_\eta(t)=0\qquad\text{for all $t\in[0,T]\,.$}
\end{equation}
Hence, letting $\eta\to0$ in \eqref{eq:QE*}, we deduce from \eqref{EST0} and \eqref{EST1} that 
\begin{equation}\label{eq:QE**}
H(u_1(t)|u_2(t))-H(u_1^{in}|u_2^{in})=T^2(t)\qquad\text{for $t\in[0,T]$\,.}
\end{equation}
With respect to $T^2(t)$, we note that
\begin{equation}\label{ott}
\begin{aligned}
\frac{T^2(t)}{a}&=-\int_0^t \int_\0  \left[ \nabla f_1\cdot \nabla \Big(f_1+\frac{b}{a}g_1\Big)  +\frac{b}{a}\nabla g_1\cdot\nabla \Big(f_1+\frac{d}{c}g_1\Big) \right]\, \mathrm{d}x \mathrm{d}s\\[1ex]
&\quad+\int_0^t \int_\0 \left[ \frac{f_1}{f_2 }\nabla f_2\cdot\nabla\Big(f_1+\frac{b}{a}g_1\Big)+\frac{b}{a}\frac{g_1}{g_2 } \nabla g_2\cdot\nabla \Big( f_1 +\frac{d}{c} g_1 \Big) \right]\, \mathrm{d}x \mathrm{d}s\\[1ex]
&\quad+\int_0^t \int_\0 \left[ \nabla f_1\cdot\nabla \Big(f_2+\frac{b}{a}g_2\Big)+\frac{b}{a} \nabla g_1\cdot\nabla \Big( f_2 +\frac{d}{c} g_2 \Big) \right] \, \mathrm{d}x \mathrm{d}s\\[1ex]
&\quad-\int_0^t \int_\0 \left[ \frac{f_1}{f_2} \nabla f_2\cdot\nabla \Big(f_2+\frac{b}{a}g_2\Big)+\frac{b}{a} \frac{g_1}{g_2}\nabla g_2\cdot\nabla \Big( f_2 +\frac{d}{c} g_2 \Big) \right] \, \mathrm{d}x \mathrm{d}s\,.
\end{aligned}
\end{equation}
Introducing
\begin{equation*}
	T_I^2(t) := - \cfrac{b(ad-bc)}{ac} \int_0^t \int_\Omega \left[ |\nabla g_1|^2 - \left( 1 + \frac{g_1}{g_2} \right) \nabla g_1\cdot \nabla g_2 + \frac{g_1}{g_2} |\nabla g_2|^2 \right]\, \mathrm{d}x \mathrm{d}s
\end{equation*}
and $T_{II}^2(t) := T^2(t) - T_I^2(t)$, we note that
\begin{equation}  \label{TI}
\begin{aligned}
	T^2_I(t)& = - \frac{b (ad-bc)}{ac}\int_0^t \int_\0 \left[  \Big|\nabla g_1-\frac{1}{2}\Big(1+\frac{g_1}{g_2 }\Big) \nabla g_2\Big|^2  
	-\Big|   \frac{g_1-g_2}{2g_2 }  \nabla g_2\Big|^2  \right] \, \mathrm{d}x \mathrm{d}s \\[1ex]
		&\leq \frac{b (ad-bc)}{ac}\int_0^t \int_\0  \Big|   \frac{g_1-g_2}{2g_2 }  \nabla g_2\Big|^2  \, \mathrm{d}x\mathrm{d}s\,,
	\end{aligned}
	\end{equation}
thanks to \eqref{condabcd}. Furthermore, in view of the relation
\[
\frac{d}{c}=\frac{b}{a}+\frac{ad-bc}{ac}\,,
\]
\begin{align*}
\frac{T^2_{II}(t)}{a}&=-\int_0^t \int_\0  \Big| \nabla \Big(f_1+\frac{b}{a}g_1\Big)\Big|^2 \, \mathrm{d}x \mathrm{d}s \\[1ex]
&\quad - \int_0^t \int_\Omega \nabla\Big(f_1+\frac{b}{a}g_1\Big)\cdot \Big[\Big(1+\frac{f_1}{f_2 }\Big)\nabla f_2 +\frac{b}{a}\Big(1+\frac{g_1}{g_2 } \Big)\nabla g_2\Big]\, \mathrm{d}x\mathrm{d}s\\[1ex]
&\quad-\int_0^t \int_\0 \nabla \Big(f_2+\frac{b}{a}g_2\Big) \cdot\Big( \frac{f_1}{f_2} \nabla f_2 +\frac{b}{a} \frac{g_1}{g_2}\nabla g_2 \Big) \, \mathrm{d}x \mathrm{d}s\\[1ex]
&=-\int_0^t \int_\0  \bigg| \nabla \Big(f_1+\frac{b}{a}g_1\Big)-\frac{1}{2}\Big[\Big(1+\frac{f_1}{f_2 }\Big)\nabla f_2 +\frac{b}{a}\Big(1+\frac{g_1}{g_2 } \Big)\nabla g_2\Big]\bigg|^2\, \mathrm{d}x \mathrm{d}s\\[1ex]
&\quad+ \frac{1}{4} \int_0^t \int_\0\Big| \Big(1+\frac{f_1}{f_2 }\Big)\nabla f_2 +\frac{b}{a}\Big(1+\frac{g_1}{g_2 } \Big)\nabla g_2\Big|^2 \, \mathrm{d}x \mathrm{d}s\\[1ex]
&\quad-\int_0^t \int_\0 \nabla \Big(f_2+\frac{b}{a}g_2\Big) \cdot\Big( \frac{f_1}{f_2} \nabla f_2 +\frac{b}{a} \frac{g_1}{g_2}\nabla g_2 \Big) \, \mathrm{d}x \mathrm{d}s \,.
\end{align*}
Observing that
\begin{align*}
	& \frac{1}{4} \Big| \Big(1+\frac{f_1}{f_2 }\Big)\nabla f_2 +\frac{b}{a}\Big(1+\frac{g_1}{g_2 } \Big)\nabla g_2\Big|^2 - \nabla\Big(f_2+\frac{b}{a}g_2\Big) \cdot\Big( \frac{f_1}{f_2} \nabla f_2 +\frac{b}{a} \frac{g_1}{g_2}\nabla g_2 \Big) \\[1ex]
	&\quad = \frac{1}{4} \Big| \nabla\Big( f_2 + \frac{b}{a} g_2\Big) + \frac{f_1}{f_2 } \nabla f_2 + \frac{b}{a} \frac{g_1}{g_2} \nabla g_2\Big|^2 - \nabla\Big(f_2+\frac{b}{a}g_2\Big) \cdot\Big( \frac{f_1}{f_2} \nabla f_2 +\frac{b}{a} \frac{g_1}{g_2}\nabla g_2 \Big) \\[1ex]
	&\quad =  \frac{1}{4} \Big| \nabla\Big( f_2 + \frac{b}{a} g_2\Big) - \frac{f_1}{f_2 } \nabla f_2 - \frac{b}{a} \frac{g_1}{g_2} \nabla g_2\Big|^2\\[1ex]
	&\quad = \frac{1}{4} \Big| \Big( 1 - \frac{f_1}{f_2} \Big) \nabla f_2 + \frac{b}{a} \Big( 1 - \frac{g_1}{g_2} \Big) \nabla g_2 \Big|^2 \\[1ex]
	&\quad \le \frac{1}{2} \Big| \frac{f_1-f_2}{f_2} \nabla f_2 \Big|^2 + \frac{b^2}{2a^2} \Big| \frac{g_1-g_2}{g_2} \nabla g_2 \Big|^2\,,
\end{align*}
the last estimate resulting from Young's inequality, we are led to
\begin{equation}\label{TII}
\frac{T^2_{II}(t)}{a}\leq\frac{1}{2}\int_0^t \int_\0 \left[ \Big|\frac{f_1-f_2}{f_2 }\nabla f_2\Big|^2 +\frac{b^2}{a^2}\Big|\frac{g_1-g_2}{g_2 }\nabla g_2\Big|^2 \right]\, \mathrm{d}x \mathrm{d}s\,.
\end{equation}
On behalf of \eqref{eq:QE**}, \eqref{TI},  and \eqref{TII}  we conclude that 
\begin{align*}
H(u_1(t)|u_2(t))&\leq H(u_1^{in}|u_2^{in})+ \frac{b (ad-bc)}{ac}\int_0^t \int_\0 \Big|   \frac{g_1-g_2}{2g_2 }  \nabla g_2\Big|^2  \, \mathrm{d}x \mathrm{d}s\\[1ex]
&\quad +\frac{a}{2}\int_0^t \int_\0 \left[ \Big|\frac{f_1-f_2}{f_2 }\nabla f_2\Big|^2 +\frac{b^2}{a^2}\Big|\frac{g_1-g_2}{g_2 }\nabla g_2\Big|^2 \right]\, \mathrm{d}x \mathrm{d}s\,.
\end{align*}
Recalling \eqref{bounds}, we deduce that there exists a positive constant $C=C(a,b,c,d)$ such that
\begin{equation}\label{EX2}
H(u_1(t)|u_2(t))\leq H(u_1^{in}|u_2^{in})+ C \sigma^4 \int_0^t \int_\0 \Big[ |f_1-f_2|^2 + \frac{b}{c}|g_1-g_2|^2 \Big] \, \mathrm{d}x \mathrm{d}s  
\end{equation}
for all $t\in[0,T]$.
In view of the inequality
\begin{equation}\label{eq:L2}
x\ln \Big(\frac{x}{y}\Big)-(x-y)\geq \frac{1}{2} \frac{|x-y|^2}{\max\{x,y\}}\,, \qquad (x,y)\in[0,\infty)\times (0,\infty)\,,
\end{equation}
which follows from \cite[Lemma~18]{HJT22x}, it is not difficult to infer from \eqref{EX2}, by taking also into account the boundedness of $u_1$ and $u_2$ in $(0,T)\times\Omega$ provided by \eqref{bounds2}, that 
\begin{align}\label{EX3}
H(u_1(t)|u_2(t))&\leq H(u_1^{in}|u_2^{in}) +C \int_0^tH(u_1(s)|u_2(s))\, \mathrm{d}s\qquad\text{for all $t\in[0,T]$}\,.
\end{align}
This completes the proof of \eqref{DES}.

\appendix
 \section{A  version of the Lions-Magenes lemma}\label{Sec:A}

In this section we establish a version of the Lions-Magenes lemma, see Lemma~\ref{L:A1} below, which is used in the proof of Theorem~\ref{MT1} when differentiating the mapping
\[
\Big[t\mapsto\int_\Omega [f(t)\ln (f(t)+\eta)-f(t)]\, \mathrm{d}x\Big]:(0,T)\to\R\,,\qquad \text{with $\eta>0$\,,} 
\] 
 for some appropriate non-negative function $f$. Before stating the result, we note that the function~${\Phi(s) := s\ln{(s+\eta)} - s}$, $s\ge 0$, satisfies $\Phi''(s)=(s+2\eta)/(s+\eta)^2$, $s\ge 0$. Thus
\[
\|\Phi''\|_\infty<\infty\,.
\]

\begin{lemma}[Lions-Magenes lemma]\label{L:A1}
Let $\Omega\subset\R^n$ be a bounded  open set and~$\Phi\in {\rm C}^2(\R)$ satisfy $\|\Phi''\|_\infty<\infty.$
Assume that 
\[
f\in L_2((0,T),H^1(\0)) \cap W^1_2((0,T),H^1(\0)')\,. 
\]
Then 
\[
\Big[t\mapsto I(t):=\int_\Omega\Phi(f(t))\, \mathrm{d}x\Big]\in {\rm C}([0,T],\R)
\]
and for all $0\leq t_0\leq t\leq T$ we have
\begin{equation}\label{id1}
\int_\Omega\Phi(f(t))\, \mathrm{d}x-\int_\Omega\Phi(f(t_0))\, \mathrm{d}x=\int_{t_0}^t\Big\langle\partial_t f(\tau),\Phi'(f(\tau))\Big\rangle_{(H^1)',H^1} \mathrm{d}\tau\,.
\end{equation}
\end{lemma}

As we are lacking a precise reference for Lemma~\ref{L:A1}, we include below a proof for the sake of completeness.
 As a  first step, we establish  in Lemma~\ref{L:A0}  an auxiliary result which is used in the proof of Lemma~\ref{L:A1}.

\begin{lemma}\label{L:A0}
Let $\Omega\subset\R^n$ be a bounded  open set and let $f\in {\rm C}^1(\mathcal{I},L_2(\Omega))$, where $\mathcal{I}\subset \R$ is an interval.
Let further $\Phi\in {\rm C}^2(\R) $ satisfy $
\|\Phi''\|_\infty=:L<\infty$.
Then,
\[
\Big[t\mapsto I(t):=\int_\Omega\Phi(f(t))\, \mathrm{d}x\Big]\in {\rm C}^1(\mathcal{I},\R)
\]
and
\begin{equation}\label{z0} 
I'(t)=\frac{\mathrm{d}}{\mathrm{d}t} \int_\Omega\Phi(f(t))\, \mathrm{d}x= \int_\Omega\Phi'(f(t))\partial_t f(t)\, \mathrm{d}x,\qquad t\in\mathcal{I}\,.
\end{equation}
\end{lemma}

\begin{proof}
We may assume without loss of generality that $\Phi(0)=\Phi'(0)=0$ (as the claim is obvious for affine functions). Then 
\begin{equation}
	|\Phi(r) - \Phi(s)| \le L (|r|+|s|) |r-s|\,, \quad |\Phi'(r)-\Phi'(s)| \le L |r-s|\,, \qquad (r,s)\in\mathbb{R}^2\,. \label{z1}
\end{equation}
In particular, since $\Phi(0)=\Phi'(0)=0$,
\[
|\Phi(f(t))|\leq L|f(t)|^2\qquad\text{and}\qquad|\Phi'(f(t))|\leq L|f(t)|\,, \qquad t\in\mathcal{I}\,,
\]
and it follows that
$$
\Phi(f(t))\in L_1(\Omega) \qquad\text{and}\qquad \Phi'(f(t)) \partial_t f(t) \in L_1(\Omega)\,, \qquad t\in\mathcal{I}\,.
$$
Let $t\neq t_0\in\mathcal{I}$. We then have
\begin{align*}
&\Big|\frac{I(t)-I(t_0)}{t-t_0}-\int_\Omega\Phi'(f(t_0))\partial_t f(t_0)\, \mathrm{d}x\Big|\\[1ex]
&\leq \int_\Omega\Big|\frac{\Phi(f(t))-\Phi(f(t_0))}{t-t_0}-\Phi'(f(t_0))\partial_t f(t_0)\Big|\, \mathrm{d}x\\[1ex]
&\leq\int_\Omega \int_0^1\Big|\Phi'((1-s)f(t_0)+sf(t))\frac{f(t)-f(t_0)}{t-t_0}-\Phi'(f(t_0))\partial_t f(t_0)\Big|\, \mathrm{d}s \mathrm{d}x\\[1ex]
&\leq J_1(t)+J_2(t)\,,
\end{align*}
where
\begin{align*}
J_1(t)&:=\int_\Omega \int_0^1\Big|\Phi'((1-s)f(t_0)+sf(t))\Big[\frac{f(t)-f(t_0)}{t-t_0}- \partial_t f(t_0)\Big]\Big|\, \mathrm{d}s \mathrm{d}x\,,\\[1ex]
J_2(t)&:=\int_\Omega \int_0^1\Big|[\Phi'((1-s)f(t_0)+sf(t)) -\Phi'(f(t_0))]\partial_t f(t_0)\Big|\, \mathrm{d}s \mathrm{d}x\,.
\end{align*}
By~\eqref{z1},  H\"older's inequality, and the regularity of $f$,
\begin{align*}
J_1(t)&\leq L(\|f(t_0)\|_2+\|f(t)\|_2)\Big\|\frac{f(t)-f(t_0)}{t-t_0}- \frac{df}{dt}(t_0)\Big\|_2\underset{t\to t_0}\to 0\,,\\[1ex]
J_2(t)&\leq L\|f(t)-f(t_0)\|_2\Big\|\partial_t f(t_0)\Big\|_2\underset{t\to t_0}\to 0\,.
\end{align*}
Therefore, $I$ is differentiable at $t_0$ and its derivative is given by~\eqref{z0}. 
It next readily follows from~\eqref{z1} and the regularity of $f$ that $\Phi'(f)$ and $\partial_t f$ both belong to~${\rm C}(\mathcal{I},L_2(\Omega))$, 
from which we deduce that $I'\in {\rm C}(\mathcal{I})$ with the help of H\"older's inequality.
\end{proof}

We now recall  a basic property which is used in the proof of Lemma~\ref{L:A1} below.  Let~${X,\,Y}$ be Banach spaces such that the embedding of $X$ in $Y$ is continuous and dense and let~${T>0}$. Then, ${\rm C}^\infty([0,T],X)$ is dense in 
$$
E_{2}(X,Y) := L_2((0,T),X) \cap W^1_2((0,T),Y),
$$
see, e.g., \cite[Lemma~II.5.10]{BF13}.

\begin{proof}[Proof of Lemma~\ref{L:A1}]
Since ${\rm C}^\infty([0,T],H^1(\Omega))$ is dense in $E_2(H^1(\0),H^1(\0)')$, there is a sequence $(f_n)_{n\ge 1} \in {\rm C}^\infty([0,T],H^1(\Omega))$ such that
\begin{equation}
\lim_{n\to\infty} \|f_n-f\|_{L_2((0,T),H^1(\Omega))} = \lim_{n\to\infty} \left\|\partial_t f_n - \partial_t f \right\|_{L_2((0,T),H^1(\Omega)')} = 0\,. \label{z2}
\end{equation}
Moreover, thanks to the continuous embedding  of $E_2(H^1(\0),H^1(\0)')$ in ${\rm C}([0,T],L_2(\Omega))$, see, e.g., \cite[Theorem~II.5.13]{BF13}, we deduce from~\eqref{z2} that
\begin{equation}
	\lim_{n\to\infty} \sup_{t\in [0,T]} \|f_n(t)-f(t)\|_2 = 0\,. \label{z3}
\end{equation}
Let $0\le t_0 \le t \le T$. By Lemma~\ref{L:A0} 
\begin{equation}
\int_\Omega\Phi(f_n(t))\, \mathrm{d}x- \int_\Omega\Phi(f_n(t_0))\, \mathrm{d}x=\int_{t_0}^t\Big\langle \partial_t f_n(\tau),\Phi'(f_n(\tau))\Big\rangle_{(H^1)',H^1}\, \mathrm{d}\tau\,. \label{z4}
\end{equation}
On the one hand, we infer from~\eqref{z1}, \eqref{z3}, and H\"older's inequality that
\begin{equation}
\lim_{n\to\infty} \int_\Omega\Phi(f_n(t))\, \mathrm{d}x = \int_\Omega\Phi(f(t))\, \mathrm{d}x \;\;\text{ and }\;\; \lim_{n\to\infty} \int_\Omega\Phi(f_n(t_0))\, \mathrm{d}x = \int_\Omega \Phi(f(t_0))\, \mathrm{d}x\,. \label{z5}
\end{equation} 
On the other hand, it readily follows from~\eqref{z1} and \eqref{z2} that
$$
\lim_{n\to\infty} \int_0^T \|\Phi'(f_n(\tau)) - \Phi'(f(\tau))\|_2^2\,\mathrm{d}\tau = 0\,.
$$
Moreover, the boundedness and continuity of $\Phi''$, \eqref{z2}, and Lebesgue's dominated convergence theorem entail that 
\[
\text{$\Phi'(f_n)\in L_2((0,T),H^1(\Omega))$\quad with\quad$\nabla \Phi'(f_n) = \Phi''(f_n) \nabla f_n$\,,\qquad $n\ge 1$\,,}
\] and
$$
\lim_{n\to\infty} \int_0^T \| \Phi''(f_n(\tau)) \nabla f_n(\tau) - \Phi''(f(\tau)) \nabla f(\tau) \|_2^2\, \mathrm{d}\tau = 0\,.
$$
Therefore,
\begin{equation*}
	\lim_{n\to\infty} \|\Phi'(f_n) - \Phi'(f) \|_{L_2((0,T),H^1(\Omega))} = 0\,.
\end{equation*}
Combining this convergence with~\eqref{z2}, leads us to 
\begin{equation}
\lim_{n\to\infty} \int_{t_0}^t \Big\langle \partial_t f_n(\tau),\Phi'(f_n(\tau))\Big\rangle_{(H^1)',H^1}\, \mathrm{d}\tau = \int_{t_0}^t \Big\langle \partial_t f(\tau),\Phi(f(\tau)) \Big\rangle_{(H^1)',H^1}\, \mathrm{d}\tau = 0\,.\label{z6}
\end{equation}
The identity~\eqref{id1} is then a direct consequence of~\eqref{z4}, \eqref{z5}, and~\eqref{z6}.
\end{proof}

\subsection*{Acknowledgement}
The authors gratefully acknowledge the support by the RTG 2339
 ``Interfaces, Complex Structures, and Singular Limits'' of the German Science Foundation (DFG).

\bibliographystyle{siam}
\bibliography{LM22}
\end{document}